\documentclass[a4j,10pt,onecolumn,oneside,notitlepage,openany,final]{report}
\usepackage{amsmath}
\usepackage{amsthm}
\usepackage{amsfonts}
\usepackage{amssymb}
\usepackage{bm}
\usepackage{graphicx}
\usepackage{longtable}

\newtheorem{thm}[]{Theorem}

\newtheorem{lem}[thm]{Lemma}
\newtheorem{cor}[thm]{Corollary}
\newtheorem{prp}[thm]{Proposition}
\newtheorem{rem}[thm]{Remark}

\begin{document}
\title{\textbf{\mathversion{bold} 
Motivic Hilbert zeta functions for curve singularities and related invariants}}
\author{Masahiro WATARI}
\date{}
\maketitle

\begin{abstract}
In the present paper, 
we show that the motivic Hilbert zeta function for a curve singularity yields 
the generating functions for Euler numbers of punctual Hilbert schemes when any punctual Hilbert scheme admits an affine cell decomposition. 
This fact allows us to derive the relations among the motivic Hilbert zeta function and 
other invariants such as the generating function for semi-modules of the  semi-group of the singularity, 
the HOMFLY polynomial and the degrees of Severi strata of the miniversal deformation of the singularity.

As an application of the fact above, we also generalize Kawai's result regarding the generating function for the Euler numbers of a singular curve.
\end{abstract}
\section{Introduction} 
Let $C$ and $p$ be a reduced singular curve over $\mathbb{C}$ and its singular point respectively. 
We refer to the pair $(C,p)$ as a curve singularity.
We denote by $K_0(\mathrm{Var}_{\mathbb{C}})$ be the Grothendieck ring of $\mathbb{C}$-varieties. 
This is the free abelian group on isomorphism classes $[X]$ of $\mathbb{C}$-varieties $X$, subject to the relation 
\begin{equation}\label{G ring property 1}
[X]=[Y]+[X\setminus Y]
\end{equation}
where $Y\hookrightarrow X$ is a closed embedding. 
Its multiplication is given by
\begin{equation}\label{G ring property 2}
[X]\cdot[Y]=[(X\times Y)_{\mathrm{red}}]
\end{equation}
on classes of $\mathbb{C}$-varieties. 
For a reduced curve singularity $(C,p)$, the \emph{motivic Hilbert zeta function with support p} is defined by
\begin{equation}\label{MHZF w.s.i. Y}
Z^{\mathrm{Hilb}}_{C,p}(t):=\sum_{l= 0}^{\infty}[C^{[l]}_p]t^l\in 1+tK_0(\mathrm{Var}_{\mathbb{C}})[[t]]
\end{equation}
where $C^{[l]}_p$ consits of length $l$ subschemes of $C$ supported on $p$. 
We refer to $C^{[l]}_p$ as \emph{the punctual Hilbert scheme of degree $l$} for the singularity $(C,p)$. 

Throughout this paper, we always assume that $(C,p)$ satisfies the following condition:

\begin{center}
$(*)$ Any punctual Hilbert schemes $C^{[l]}_p$ admit an affine cell decomposition 
(i.e. $C^{[l]}_p=\bigsqcup_{i} \mathbb{A}^{n_i})$.
\end{center}

It is known that some reduced irreducible plane curve singularities with certain characteristic exponents satisfy the condition\,$(*)$ (see   Remark\,{\rm \ref{rem 10}}). 

Let $\mathbb{L}$ be the class of the affine line $\mathbb{A}^1$ in $K_0(\mathrm{Var}_{\mathbb{C}})$. 
The motivic Hilbert zeta function $Z^{\mathrm{Hilb}}_{C,p}(t)$ is an element of $\mathbb{C}[\mathbb{L}][[t]]$ under the condition\,$(*)$  (see Lemma\,\ref{main thm 1}).  
So we use the notation $Z^{\mathrm{Hilb}}_{C,p}(t, \mathbb{L})$ for it, instead of  $Z^{\mathrm{Hilb}}_{C,p}(t)$. 
Our first result in the present paper is the following theorem:

\begin{thm}\label{main 2}
Let $(C,p)$ be a reduced curve singularity. 
Under the condition $(*)$,  we have
\begin{equation*}
Z^{\mathrm{Hilb}}_{C,p}(q, 1)=\sum_{l=0}^\infty\chi (C^{[l]}_p)q^{l}
\end{equation*}
where $\chi (C^{[l]}_p)$ is the Euler number of $C^{[l]}_p$. 
\end{thm}

Let $\mathbb{N}$ be the set of natural numbers. 
We put $\mathbb{N}_0:=\mathbb{N}\cup\{0\}$. 
A nonempty subset $S$ of  $\mathbb{N}_0$ is called a \emph{numerical semigroup}, if it satisfies the following three condtions:
(1) $S$ is closed under addition, (2) $0\in S$, (3) $\# (\mathbb{N}_0\setminus S) < \infty$. 
Let $\Gamma$ and $\mathrm{Mod}(\Gamma)$ be a numerical semi-group 
and the set of all $\Gamma$-semi-moudules respectively. 
The \emph{generating function of $\Gamma$-semi-modules} $I(\Gamma;q)$ is defined to be 
$$
I(\Gamma;q):=\sum_{\Delta\in \mathrm{Mod}(\Gamma)}q^{\mathrm{codim}(\Delta)}
$$
where  $\mathrm{codim}(\Delta):=\sharp\{\Gamma \setminus \Delta\}$. 
Later we define a numerical semi-group associated with a reduced irreducible curve singularity. 
\begin{thm}\label{main 4}
Let $(C,p)$ be a reduced irreducible curve singularity whose semi-group is $\Gamma$. 
Under the condition $(*)$,  we have
\begin{equation*}
Z^{\mathrm{Hilb}}_{C,p}(q, 1)=I(\Gamma; q).
\end{equation*}
\end{thm}

In the rest of this section, we restrict ourselves to considering reduced irreducible plane curve singularities. 
In this situation, we argue the relations between $Z^{\mathrm{Hilb}}_{C,p}(t, \mathbb{L})$ and other invariants. 
It is known that the generating function for the Euler numbers of punctual Hilbert schemes is equivalent to other invariants. 
Let $P(L_{(C,p)})$ be the HOMFLY polynomial of the oriented link $L_{(C,p)}$ associated with $(C,p)$. 
The following relation was conjectured by Oblomkov and Shende in \cite{OV} and was finally proved by Maulik in \cite{Ma}:

\begin{equation}\label{Maulik}
\sum_{l=0}^\infty\chi (C^{[l]}_p)q^{2l}=\left( \frac{q}{a}\right)^{\mu-1}P(L_{(C,p)})\Big|_{a=0}
\end{equation}


Shende \cite{S} also  proved the relation
\begin{equation}\label{Severi}
\sum_{l=0}^\infty\chi (C^{[l]}_p)q^{l}=\sum_{l=0}^\delta q^{\delta-l}(1-q)^{2h-1}\mathrm{deg}_{p}\mathbb{V}_h
\end{equation}
where $(C,p)$ is a plane curve singularity (which may not be irreducible) with its delta invariant $\delta$ and $\mathbb{V}_h$'s are the severi strata of the miniversal deformation of $(C,p)$.

Consequently, the following fact follows from Theorem\,\ref{main 2}, \ref{main 4}, (\ref{Maulik}) and (\ref{Severi}). 

\begin{thm}\label{main 3}
Here notations remain the same as above. 
If  $(C,p)$ is a reduced irreducible plane curve singularity that satisfies  the condition $(*)$,  we have
\begin{align}\label{H}
&Z^{\mathrm{Hilb}}_{C,p}(q^2, 1)=I(\Gamma;q^2)=\left( \frac{q}{a}\right)^{\mu-1}P(L_{(C,p)})\Big|_{a=0},\\ \label{S}
&Z^{\mathrm{Hilb}}_{C,p}(q, 1)=I(\Gamma;q)=\sum_{l=0}^\delta q^{\delta-l}(1-q)^{2h-1}\mathrm{deg}_{p}\mathbb{V}_h.
\end{align}
\end{thm}

\begin{rem}
The equivalence of the HOMFLY polynomial and the generating function of $\Gamma$-semi-modules $I(\Gamma;q)$ in {\rm (\ref{H})} was pointed out by Chavan {\rm (\cite{C})}. 
\end{rem}

This paper is organized as follows:
The purpose of Section\,\ref{preliminary} is to prove Theorem\,\ref{main 2} and \ref{main 4}. 
In Subection\,\ref{proof 1}, we first prove some elementary facts regarding $C_p^{l}$. 
After those, we prove Theorem\,\ref{main 2}. 
In Subsection\,\ref{preliminaries 2}, for the proof of Theorem\,\ref{main 4}, we briefly recall basic notions and facts about punctual Hilbert schemes for reduced irreducible curve singularities. 
In Subection\,\ref{proof 2}, we prove Theorem\,\ref{main 4}.  
We also derive a property of  $I(\Gamma;q)$ from Thorem\ref{main 4}. 
In Section\,\ref{examples}, we calculate motivic Hilbert zeta functions for the curve singularities of types $A_1$, $A_{2d}$, $E_6$ and $E_8$, as examples. 
In Section\,\ref{application}, we consider another topic. 
The generating function of Euler numbers of symmetric products of a nonsingular projective complex curve was calculated. 
Later, Kawai obtained a similar result for an integral complex projective singular curve with the $A_1$ and $A_2$-singularities. 
We generalize Kawai's to the case of an integral complex projective singular curve with the $A_1$, $A_{2i}$ $(i=1,\ldots,d)$ $E_6$ and $E_8$-singularities.

\section{Proof of main theorems}\label{preliminary}
In this section, we prove Theorem\,\ref{main 2} and \ref{main 4}. 
\subsection{Proof of Theorem\,\ref{main 2}}\label{proof 1}

We need the following properties of Euler numbers of complex varieties:
\begin{align}\label{property of chi 1}
&\bullet \chi(X\setminus Y)=\chi(X)- \chi(Y) \mathrm{\ for\ a\ closed\ immersion\ } Y\subset X,\\ \label{property of chi 2}
&\bullet \chi(X\times Y)=\chi(X)\chi(Y),\\ \label{property of chi 3}
&\bullet \chi(\mathbb{A}^1)=1 
\end{align}

We first prove the following lemma:
\begin{lem}\label{main thm 1}
Let $(C,p)$ be a reduced curve singularity that satisfies the condition $(*)$.  
The class $[C^{[l]}_p]$ of $C^{[l]}_p$ in $K_0(\mathrm{Var}_{\mathbb{C}})$ is a polynomial in $\mathbb{C}[\mathbb{L}]$.  
Furthermore, the Euler number $\chi (C^{[l]}_p)$ is equal to the number of affine cells of $C^{[l]}$. 
\end{lem}
\begin{proof}
Let $C^{[l]}_p=\sqcup_{i=1}^{s_l} \mathbb{A}^{n_{i}}$ be an affine cell decomposition of $C^{[l]}_p$. 
It follows from (\ref{G ring property 1}) and (\ref{G ring property 2}) that
\begin{equation*}
[C^{[l]}_p]=\Big [ \bigsqcup_{i=1}^{s_l} \mathbb{A}^{n_{i}} \Big ]=\sum_{i=1}^{s_l} \left [\mathbb{A}^{n_{i}}\right ]=\sum_{i=1}^{s_l} \mathbb{L}^{n_{i}}.
\end{equation*}
It also follows from (\ref{property of chi 1}), (\ref{property of chi 2}) and (\ref{property of chi 3}) that
\begin{equation*}
\chi (C^{[l]}_p)=\chi\Big ( \bigsqcup_{i=1}^{s_l} \mathbb{A}^{n_{i}} \Big )=\sum_{i=1}^{s_l}\chi \left (\mathbb{A}^{n_{i}}\right )=s_l.
\end{equation*}
Our assertions are proved. 
\end{proof}
Setting $f^{(l)}_{C,p}(\mathbb{L}):=[C^{[l]}_p]$ under the condition $(*)$,  
 Lemma\,\ref{main thm 1} implies the following fact:
\begin{prp}\label{prp 1}
Let $(C,p)$  be a reduced curve singularity. 
Under the condition $(*)$, we have
$$
f^{(l)}_{C,p}(1)=\chi(C^{[l]}_p).
$$
\end{prp}

\noindent
\textit{Proof of Theorem\,$\ref{main 2}$.}
It follows from (\ref{MHZF w.s.i. Y}) and Lemma\,\ref{main thm 1} that
$$
Z^{\mathrm{Hilb}}_{C,p}(t, \mathbb{L})=\sum _{l=0}^{\infty}f^{(l)}_{C,p}(\mathbb{L})t^l.
$$
Finally, our assertion follows from Propostion\,\ref{prp 1}.  $\square$

\subsection{Preliminaries for the proof of Theorem\,\ref{main 4}}\label{preliminaries 2}
In the rest of this section, we always assume that $(C,p)$ is reduced and irreducible. 
For the proof of Theorem\,\ref{main 4}, we briefly summarize the properties of punctual Hilbert schemes for $(C,p)$.

Let $\mathcal{O}_{C,p}$ be  the local ring of $C$ at $p$. 
We denote by  $R:=\widehat{\mathcal{O}}_{C,p}$ the completion of  $\mathcal{O}_{C,p}$. 
Note that the structure of $C^{[l]}_p$ depends only on 
$R$. 
In fact, there is the following natural identification:

$$
C^{[l]}_p=\{I\subset R|\,\text{$I$ is an ideal of $R$ with $\dim_{\mathbb{C}} R/I=l$}\}
$$ 

Note that the normalization $\overline{R}$ of $R$ is isomorphic to $k[[t]]$. 
We call $\Gamma:=\{\text{ord}_t(f)\,|\,f\in R\}$ the \emph{semi-group} of $R$ (or of singularity $(C,p)$). 
We denote a semi-group by $\Gamma=\langle\gamma_1,\gamma_2,\cdots,\gamma_p\rangle$, if it is minimally generated by $\gamma_1,\gamma_2,\cdots,\gamma_p$.
The positive integer $\delta:=\dim _{\mathbb{C}}(\overline{R}/R)$ is called the \emph{$\delta$-invariant} of $R$.  
The \emph{conductor} $c$ of $\Gamma$ is the element of $\Gamma$ with the properties $c-1\notin \Gamma$ and $c+n\in \Gamma$ for any $n\in \mathbb{N}_0$. 

For a positive integer $a$, we denote by $(t^a)$  an ideal of $\overline{R}$.
Let $\mathrm{Gr}\left(\delta,\overline{R}/(t^{2\delta})\right)$ be the Grassmannian which consists of $\delta$-dimensional linear subspaces of $\overline{R}/(t^{2\delta})$.
For $M\in\mathrm{Gr}\left(\delta,\overline{R}/(t^{2\delta})\right)$, we define a multiplication by $R\times M\ni (f,m+(t^{2\delta}))\mapsto fm+(t^{2\delta})\in M$. 
The set
$$
J_{C,p}:=\left\{M\in \mathrm{Gr}\left(\delta,\overline{R}/(t^{2\delta})\right)\big|\,\text{$M$ is an $R$-sub-module w.r.t the above multiplication}\right\}
$$
is called the \emph{Jacobian factor} of $(C,p)$, which was introduced by Rego in \cite{Re}. 
Observe that any element $I$ in $C^{[l]}_p$ satisfies $(t^{l+2\delta}) \cap R \subseteq I\subseteq (t^l)\cap R$. 
It follows that $(t^{2\delta})\subseteq t^{-l}I \subseteq \overline{R}$ and $\dim_{\mathbb{C}}(\overline{R}/t^{-l}I)=\dim_{\mathbb{C}}(\overline{R}/I)-l=\delta$. 
Pfister and Steenbrink \cite{PS} defined a map 
\begin{equation}\label{phai}
\varphi_l: C^{[l]}_p\rightarrow J\subset \mathrm{Gr}(\delta, \overline{R}/(t^{2\delta}))
\end{equation}
by $\varphi _l (I)=t^{-l}I/(t^{2\delta})$. 
We call $\varphi _l$ the \emph{$\delta$-normalized embedding}. 
It has the following properties:

\begin{prp}[\cite{PS}, Theorem\,3]\label{thm3}
The $\delta$-normalized embedding $\varphi_l$  is injective for any non-negative integer $l$. 
It is also bijective for $l\ge c$. 
The image $\varphi _l (C^{[l]}_p)$  is Zariski closed in the Jacobian factor $J_{C,p}$. 
\end{prp}
Since $\varphi_l$ is injective, we always identify $C^{[l]}_p$ with its image by $\varphi_l$. 
We infer the following fact from Proposition\,\ref{thm3}. 

\begin{cor}\label{stable}
For $l\ge c$,  we have $[J_{C,p}]=[C^{[l]}_p]$ in $K_0(\mathrm{Var}_{\mathbb{C}})$. 
\end{cor}

Let $\Gamma$ be the semi-group of $R$.  
A subset $\Delta\subset \mathbb{N}_0$ is called a \emph{$\Gamma$-semi-module}, if the relation $\Delta+\Gamma\subseteq \Delta$ holds.   
It is obvious that, for any ideal $I$ of $R$, the set $\Gamma(I):=\{\text{ord}_t(f)|\,f\in I\}$ is a $\Gamma$-semi-module.
For a $\Gamma$-semi-module $\Delta$, the set
$
\mathcal{I}(\Delta):=\{I\subset R\,|\,\text{$I$ is an ideal of $R$ with $\Gamma(I)=\Delta$}\}
$
is called the \emph{$\Delta$-subset}. 
Put
$
\mathrm{Mod}_l(\Gamma):=\{\Delta|\, \text{$\Delta$ is a $\Gamma$-semi-module with codim $\Delta=l$}\}. 
$
It is known that a punctual Hilbert scheme $C^{[l]}_p$ for irreducible curve singularity is stratified by $\Delta$-subsets. 

\begin{prp}[\cite{SW1}, Proposition\,6]\label{decomposition lem}
Let $(C,p)$ be a reduced irreducible curve singularity.
If $\mathrm{Mod}_l(\Gamma)=\{\Delta_{l,1},\ldots,\Delta_{l,s_l}\}$, then a punctual Hilbert scheme $C^{[l]}_p$ is decomposed as
\begin{equation}\label{decomposition}
C^{[l]}_p=\bigsqcup_{i=1}^{s_l} \mathcal{I}(\Delta_{l,i}).
\end{equation}
\end{prp}
\begin{rem}\label{rem 10}
The stratification\,{\rm (\ref{decomposition})} is an affine cell decomposition for the Jacobian factors $($i.e. $C^{[l]}_p$ with $l\ge c$$)$ of reduced  irreducible plane curve singularities with the characteristic exponents $(p,q)$, $(4,2q,s)$, $(6,8,s)$ and $(6,10,s)$ $($see {\rm \cite{P}}$)$. 
This fact was generalized to any punctual Hilbert schemes of reduced irreducible plane curve singularities with the characteristic exponents $(p,q)$ in {\rm \cite{ORS}}.
\end{rem}
\subsection{Proof of Theorem\,\ref{main 4}}\label{proof 2}
In this subsection, we give the proof of Theorem\,\ref{main 4}. 
Below, we freely use the notations introduced in the previous sections.

\noindent
\textit{Proof of Theorem\,$\ref{main 4}$.} 
Note that the condition\,$(*)$ implies each $\Delta$-subset in (\ref{decomposition}) is isomorphic to an affine space. 
So the number of affine cells of $C^{[l]}$ is equals to $\sharp \mathrm{Mod}_l (\Gamma)$. 
Furthermore, we have $\sharp \mathrm{Mod}_l (\Gamma)=\chi (C_p^{l})$ by Lemma\,\ref{main thm 1}. 
It follows from Theorem\,\ref{main 2} that
\begin{align*}
I(\Gamma ;q)=\sum_{l=0}^{\infty}\sum_{\Delta\in \mathrm{Mod}_l(\Gamma)}q^{\mathrm{codim}(\Delta)}
=\sum_{l=0}^{\infty} \sharp \mathrm{Mod}_l (\Gamma)  q^{l}
=\sum_{l=0}^{\infty} \chi({C^{[l]}})   q^{l}
=Z^{\mathrm{Hilb}}_{C,p}(q, 1).
\end{align*}
$\square$

Bejleri, Rangnathan and Vakil \cite{BRV} proved that the motivic Hilbert function for a reduced curve singularity is rational.

\begin{thm}[\cite{BRV}, Theorem 2.3]\label{thm BRV}
Let  $(C,p)$ be a reduced curve singularity $($which may not satisfy the condition\,$(*))$.
The motivic Hilbert zeta function  $Z^{\mathrm{Hilb}}_{C,p}(t)$  is of the form
\begin{equation}\label{zeta}
Z^{\mathrm{Hilb}}_{C,p}(t)=\frac{f(t)}{(1-t)^b}
\end{equation}
where  $f(t)$ is a polynomial in $1+tK_0(\mathrm{Var}_{\mathbb{C}})$ and $b$ is the number of branches at $p$.
\end{thm}

\begin{rem}\label{main}
Generalizing the stability of $[C_p^{[l]}]$ in Corollary\,{\rm \ref{stable}} to the case of reducible curve singularities, Theprem\,{\rm \ref{thm BRV}} was proved in {\rm  \cite{BRV}}. 
\end{rem}

In our situations, Theorem\,\ref{thm BRV} is stated as follows:

\begin{cor}\label{cor of thm BRV}
In Theorem\,{\rm\ref{thm BRV}},  if a reduced irreducible curve singularity $(C,p)$ whose  semi-group is $\Gamma$ satisfies  the condition $(*)$, then the  polynomial $f(t)$ in {\rm (\ref{zeta}) } has the follwoing properties$:$
$({\rm i}):f(t)\in 1+\mathbb{C}[\mathbb{L}][t]$, $({\rm ii}):\deg_tf(t)=c$. 
Furthermore, if $(C,p)$ is planar , then the condition\,$({\rm ii})$ becomes $\deg_tf(t)=2\delta=\mu$ where $\mu$ is  the Milnor number of $(C,p)$.
\end{cor}

\begin{proof}
It follows from Corollary\,\ref{stable} that
\begin{align*}
Z_{C,p}^{\mathrm{Hilb}}(t, \mathbb{L})
&=\sum_{l=0}^{c-1}f_{C,p}^{(l)}(\mathbb{L})t^l+\sum_{l\ge c}^{\infty}f_{C,p}^{(l)}(\mathbb{L})t^l
=\sum_{l=0}^{c-1}f_{C,p}^{(l)}(\mathbb{L})t^l+\frac{f_{C,p}^{(c)}(\mathbb{L})t^c}{1-t}\\
&=\frac{1}{1-t}\left\{f_{C,p}^{(0)}(\mathbb{L})+\sum_{l=1}^{c-1}(f_{C,p}^{(l)}(\mathbb{L})-f_{C,p}^{(l-1)}(\mathbb{L}))t^{l}+(f_{C,p}^{(c)}(\mathbb{L})-f_{C,p}^{(c-1)}(\mathbb{L}))t^c\right\}.
\end{align*}
Since $C_p^{[0]}=\{R\}$, we have $f_{C,p}^{(0)}=1$. Furthermore, we can show $C_p^{[c-1]}\subsetneq J\cong C_p^{[c]}$. 
Hence we conclude that $f_{C,p}^{(c)}(\mathbb{L})-f_{C,p}^{(c-1)}(\mathbb{L})\neq 0$. 
In particular, it is well known that $c=\mu$ holds for a plane curve singularity (cf. \cite{Ri}, Theorem\,1). 
\end{proof}
In Corollary\,\ref{cor of thm BRV}, we use the notation $f(t,\mathbb{L})$ instead of $f(t)$. 
It follows from Corollary\,\ref{cor of thm BRV} that $Z^{\mathrm{Hilb}}_{C,p}(t,\mathbb{L})=f(t,\mathbb{L})/(1-t)$.
The following fact is the direct consequence of Thorem\,\ref{main 4}:
\begin{thm}\label{semi-moudle generating function}
For  a numerical semi-group $\Gamma$, let $(C,p)$ be a reduced irreducible curve singularity whose local ring is $\mathbb{C}[[t^\gamma|\, \gamma\in \Gamma]]$.
If $(C,p)$ satisfy the condition\,$(*)$, then the $\Gamma$-semi-module generating function $I(\Gamma;q)$ is of the form
\begin{equation}
I(\Gamma;q)=\frac{f(q,1)}{(1-t)}
\end{equation}
where  $f(q,1)\in 1+\mathbb{Z}[q]$ and $\deg_q f(q,1)=c$. 
\end{thm}
\begin{rem}
Theorem\,{\rm \ref{semi-moudle generating function}} was also argued by Chavan {\rm \cite{C}}. 
We can prove Theorem\,{\rm \ref{semi-moudle generating function}} directly, without going through a reduced irreducible curve singularity which satisfies the condition\,$(*)$. 
Theorem\,{\rm \ref{semi-moudle generating function}} holds for any numerical semi-group.  
\end{rem}

\section{Examples of motivic Hilbert zeta functions}\label{examples}
We consider some examples of motivic Hilbert zeta functions in this section. 
These examples will be used in a generalization of Kawai's result in the next section. 
\subsection{The \mathversion{bold}{$A_1$} and \mathversion{bold}{$A_{2d}$}-singularities}
Here we consider the motivic Hilbert zeta functions for the curve singularities of the $A_1$ and $A_{2d}$-types as examples.  
It is known that these singularities satisfy the condition $(*)$. 
\\
\textbf{The \mathversion{bold}$A_{1}$-singularity:}
This example can be found in \cite{Bejleri}.
Let $(C, p)$ be the $A_1$ singularity (i.e. an ordinary node).
We have $R=\mathbb{C}[[x,y]]/(xy)$. 
In this case, ideals in $R$ were studied in \cite{R}. 
It is easy to see that $C^{[0]}_p=\{R\}$ and $C^{[1]}_p=\{\mathfrak{m}\}$ where $\mathfrak{m}$ is the maximal ideal of $R$. 
For $l\ge 2$, the ideals in $C^{[l]}_p$ are given by
\begin{align*}
I^{l}_i(u_i)&=(y^i+u_ix^{l-i})\quad (u_i\in\mathbb{C}^{\times},\ i=1,\ldots, l-1),\\
Q^{l}_i&=(x^{l-i+1}, y^i)\quad (i=1,\dots, l)
\end{align*}
with the relations $\lim_{u_i\rightarrow 0}I^{l}_{i}(u_i)=Q^{l}_i$ and  $\lim_{u_i\rightarrow \infty}I^{l}_{i}(u_i)=Q^{l}_{i+1}$. 
Hence $C^{[l]}_p$ with $l\ge 2$ is a chain of $l-1$ rational curves. 
It follows that $[C^{[l]}_p]=(l-1)\mathbb{L}+1$. 
Putting these facts together, we finally obtain
\begin{equation}\label{MHZ for A_1}
Z_{C,p}^{\mathrm{Hilb}}(t, \mathbb{L})=\frac{1-t+\mathbb{L}t^2}{(1-t)^2}. 
\end{equation}

\noindent
\textbf{The \mathversion{bold}$A_{2d}$-singularity:}
We consider the $A_{2d}$-singularity where $d$ is a positive integer
(i.e. the plane curve singularity with $R=k[[t^2,t^{2d+1}]]$). 
For this singularity, we have $\Gamma=\langle 2, 2d+1\rangle$, $\delta=d$ and $c=2d$. 
To emphasis $d$, we use the notation $J^{(d)}$ for the Jacobian factor of the $A_{2d}$-singularity.
 We studied the properties of punctual Hilbert schemes of the $A_{2d}$-singularity in \cite{SW2}.

\begin{thm}[\cite{SW2}, Theorem\, 1]\label{main thm 1 in SW}
Let $d$ and $l$ be two integers with $0\le l\le 2d$. 
For the greatest integer $s$ which is not greater than $l/2$, the punctual Hilbert scheme $C_p^{[l]}$ is a rational projective variety of dimension $s$.
If $l=0, 1$, then it consists of one point. 
On the other hand, if $l\ge2$, then it is isomorphic to $J^{(s)}$. 
\end{thm}

By Theorem\,\ref{main thm 1 in SW}, it is enough to consider the Jacob factors to analyze the structures of punctual Hilbert schemes for this case. 
In \cite{SW2}, we showed that the stratification\,(\ref{decomposition}) of $J^{(d)}$ is the following  form:
\begin{equation}\label{Structure of Hilb}
J^{(d)}=\bigsqcup^d_{i=0}\mathbb{A}^i.
\end{equation}

It follows from Theorem\,\ref{main thm 1 in SW} and  (\ref{Structure of Hilb}) that
\begin{align}\notag
Z_{C,p}^{\mathrm{Hilb}}(t,\mathbb{L})&=\sum_{l\ge 0}[C_p^{[l]}]t^l\\\notag
&=[J^{[0]}]+[J^{[0]}]t+[J^{[1]}]t^2+[J^{[1]}]t^3+\cdots [J^{[l]}]t^{2l}+[J^{[l]}]t^{2l+1}+\cdots \\\notag
&=\sum_{l=0}^{d-1}[J^{(l)}](t^{2l}+t^{2l+1})+\sum_{l\ge 2d}[J^{(c)}]t^l\\\notag
&=\frac{1}{1-t}\left\{(1-t)\sum_{i=0}^{d-1}[J^{(l)}](t^{2l}+t^{2l+1})+[J^{(d)}]t^{2d}\right\}\\\notag
&=\frac{1}{1-t}\left\{\sum_{l=0}^{d-1}[J^{(l)}](t^{2l}-t^{2l+2})+[J^{(d)}]t^{2d}\right\}\\\notag
&=\frac{1}{1-t}\left\{\sum_{l=0}^{d-1}\sum^l_{i=0}\mathbb{L}^i(t^{2l}-t^{2l+2})+\sum^d_{i=0}\mathbb{L}^it^{2d}\right\}\\\label{MHZ for A_{2d}}
&=\frac{1}{1-t}\left(\sum_{i=0}^{d}\mathbb{L}^it^{2i}\right).
\end{align}

\subsection{The \mathversion{bold}{$E_6$} and \mathversion{bold}{$E_{8}$}-singularities}
In this subsection, we calculate the motivic Hilbert zeta functions for the $E_6$ and $E_{8}$-singularities. 
In \cite{SW1}, the sets $\mathrm{Mod}_l(\Gamma)$ $(l=1,\dots,c)$ are determined for both singularities. 
Furthermore, in the same paper, we also showed that the singularities of types $E_6$ and $E_8$ satisfy the condition~\,$(*)$ 
(i.e. for any $l$, all $\Delta$-subsets in the decomposition\,(\ref{decomposition}) are isomorphic to affine spaces).
Below, we listed these results as Table\,1, 2, 3 and 4. 
Finally, we derive the motivic Hilbert zeta functions from them. 

\noindent
\textbf{The \mathversion{bold}{$E_6$}-singularity:}
Let $(X,o)$ be the $E_6$-singularity (i.e. the curve singularity with $R=k[[t^3,t^4]]$). 
For this singularity, we have $\Gamma=\langle 3,4\rangle$, $\delta=3$ and $c=6$. 

\begin{center}
\begin{tabular}{c|l}
$l$&All elements of $\mathrm{Mod}_l(\Gamma)$\\
\hline
0&$\Delta_{0,1}=\langle 0\rangle_\Gamma$\\
\hline
1&$\Delta_{1,1}=\langle 3,4\rangle_\Gamma$\\
\hline
2&$\Delta_{2,1}=\langle 4,6\rangle_\Gamma$, $\Delta_{2,2}=\langle 3,8\rangle_\Gamma$\\
\hline
3&$\Delta_{3,1}=\langle 6,7,8\rangle_\Gamma$, $\Delta_{3,2}=\langle 4,9\rangle_\Gamma$, $\Delta_{3,3}=\langle 3\rangle_\Gamma$\\
\hline
4&$\Delta_{4,1}=\langle 7,8,9\rangle_\Gamma$, $\Delta_{4,2}=\langle 6,8\rangle_\Gamma$, $\Delta_{4,3}=\langle 6,7\rangle_\Gamma$, $\Delta_{4,4}=\langle 4\rangle_\Gamma$\\
\hline
5&$\Delta_{5,1}=\langle 8,9,10\rangle_\Gamma$, $\Delta_{5,2}=\langle 7,9\rangle_\Gamma$, $\Delta_{5,3}=\langle 7,8\rangle_\Gamma$, $\Delta_{5,4}=\langle 6,11\rangle_\Gamma$\\
\hline
6&$\Delta_{6,1}=\langle 9,10,11\rangle_\Gamma$, $\Delta_{6,2}=\langle 8,10\rangle_\Gamma$, $\Delta_{6,3}=\langle 8,9\rangle_\Gamma$, 
$\Delta_{6,4}=\langle 7,12\rangle_\Gamma$\\
&$\Delta_{6,5}=\langle 6\rangle_\Gamma$\\
\multicolumn{2}{c}{\textbf {Table\,}{\mathversion{bold}$1$}}
\end{tabular} 
\end{center}
All $\Delta$-subsets for $\Gamma$-modules in Table\,1 are isomorphic to affine spaces. \cite{SW1}. 
\begin{center}
\begin{tabular}{c|l}
$l$&$\Delta$-subsets in  the decomposition\,(\ref{decomposition})\\
\hline
0&$\mathcal{I}(\Delta_{0,1})\cong\mathbb{A}^0$\\
\hline
1&$\mathcal{I}(\Delta_{1,1})\cong\mathbb{A}^0$\\
\hline
2&$\mathcal{I}(\Delta_{2,1})\cong\mathbb{A}^0$, $\mathcal{I}(\Delta_{2,2})\cong \mathbb{A}^1$\\
\hline
3&$\mathcal{I}(\Delta_{3,1})\cong \mathbb{A}^0$, $\mathcal{I}(\Delta_{3,2})\cong \mathbb{A}^1$, $\mathcal{I}(\Delta_{3,3})\cong \mathbb{A}^2$\\
\hline
4&$\mathcal{I}(\Delta_{4,1})\cong  \mathbb{A}^0$, $\mathcal{I}(\Delta_{4,2})\cong  \mathbb{A}^1$, $\mathcal{I}(\Delta_{4,3})\cong  \mathbb{A}^2$, 
$\mathcal{I}(\Delta_{4,4})\cong \mathbb{A}^2$\\
\hline
5&$\mathcal{I}(\Delta_{5,1})\cong \mathbb{A}^0$, $\mathcal{I}(\Delta_{5,2})\cong \mathbb{A}^1$, $\mathcal{I}(\Delta_{5,3})\cong \mathbb{A}^2$, 
$\mathcal{I}(\Delta_{5,4})\cong \mathbb{A}^2$\\
\hline
6&$\mathcal{I}(\Delta_{6,1})\cong \mathbb{A}^0$, $\mathcal{I}(\Delta_{6,2})\cong \mathbb{A}^1$, $\mathcal{I}(\Delta_{6,3})\cong \mathbb{A}^2$, 
$\mathcal{I}(\Delta_{6,4})\cong \mathbb{A}^2$, $\mathcal{I}(\Delta_{6,5})\cong \mathbb{A}^3$\\
\multicolumn{2}{c}{\textbf {Table\,}{\mathversion{bold}$2$}}
\end{tabular} 
\end{center}
The motivic Hilbert zeta function follows from Table\,2. 
\begin{align}\notag
Z_{C,p}^{\mathrm{Hilb}}(t,\mathbb{L})=&1+t+(1+\mathbb{L})t^2+(1+\mathbb{L}+\mathbb{L}^2)t^3+(1+\mathbb{L}+2\mathbb{L}^2)t^4+(1+\mathbb{L}+2\mathbb{L}^2)t^5\\\notag
&+(1+\mathbb{L}+2\mathbb{L}^2+\mathbb{L}^3)t^6\sum_{i=0}^{\infty}t^i\\\label{MHZ for E_6}
=&\frac{1}{1-t}(1+\mathbb{L}t^2+\mathbb{L}^2t^3+\mathbb{L}^2t^4+\mathbb{L}^3t^6)
\end{align}

\noindent
\textbf{The \mathversion{bold}$E_8$-singularity:}
Let $(X,o)$ be the $E_8$-singularity (i.e. the curve singularity with $R=k[[t^3,t^5]]$). 
We have $\Gamma=\langle 3,5\rangle$, $\delta=4$ and $c=8$. 

\begin{center}
\begin{tabular}{c|l}
$l$&All elements of $\mathrm{Mod}_l(\Gamma)$\\
\hline
0&$\Delta_{0,1}=\langle 0\rangle_\Gamma$\\
\hline
1&$\Delta_{1,1}=\langle 3,5\rangle_\Gamma$\\
\hline
2&$\Delta_{2,1}=\langle 5,6\rangle_\Gamma$, $\Delta_{2,2}=\langle 3,10\rangle_\Gamma$\\
\hline
3&$\Delta_{3,1}=\langle 6,8,10\rangle_\Gamma$, $\Delta_{3,2}=\langle 5,9\rangle_\Gamma$, $\Delta_{3,3}=\langle 3\rangle_\Gamma$\\
\hline
4&$\Delta_{4,1}=\langle 8,9,10\rangle_\Gamma$, $\Delta_{4,2}=\langle 6,10\rangle_\Gamma$, $\Delta_{4,3}=\langle 6,8\rangle_\Gamma$, $\Delta_{4,4}=\langle 5,12\rangle_\Gamma$\\
\hline
5&$\Delta_{5,1}=\langle 9,10,11\rangle_\Gamma$, $\Delta_{5,2}=\langle 8,10,12\rangle_\Gamma$, $\Delta_{5,3}=\langle 8,9\rangle_\Gamma$, $\Delta_{5,4}=\langle 6,13\rangle_\Gamma$\\
&$\Delta_{5,5}=\langle 5\rangle_\Gamma$\\
\hline
6&$\Delta_{6,1}=\langle 10,11,12\rangle_\Gamma$, $\Delta_{6,2}=\langle 9,11,13\rangle_\Gamma$, $\Delta_{6,3}=\langle 9,10\rangle_\Gamma$, $\Delta_{6,4}=\langle 8,12\rangle_\Gamma$\\ &$\Delta_{6,5}=\langle 8,10\rangle_\Gamma$, $\Delta_{6,6}=\langle 6\rangle_\Gamma$\\
\hline
7&$\Delta_{7,1}=\langle 11,12,13\rangle_\Gamma$, $\Delta_{7,2}=\langle 10,12,14\rangle_\Gamma$, $\Delta_{7,3}=\langle 10,11\rangle_\Gamma$\\
&$\Delta_{7,4}=\langle 9,13\rangle_\Gamma$, $\Delta_{7,5}=\langle 9,11\rangle_\Gamma$, $\Delta_{7,6}=\langle 8,15\rangle_\Gamma$\\
\hline
8&$\Delta_{8,1}=\langle 12,13,14\rangle_\Gamma$, $\Delta_{8,2}=\langle 11,13,15\rangle_\Gamma$, $\Delta_{8,3}=\langle 11,12\rangle_\Gamma$\\
&$\Delta_{8,4}=\langle 10,14\rangle_\Gamma$, $\Delta_{8,5}=\langle 10,12\rangle_\Gamma$, $\Delta_{8,6}=\langle 9,16\rangle_\Gamma$,  $\Delta_{8,7}=\langle 8\rangle_\Gamma$\\
\multicolumn{2}{c}{\textbf {Table\,}{\mathversion{bold}$3$}}
\end{tabular} 
\end{center}

Same as the case of the $E_8$-singularity, all $\Delta$-subsets for $\Gamma$-semi-moudle in Table\,3 are isomorphic to affine spaces. 
\begin{center}
\begin{tabular}{c|l}
$l$&Affine cells in the decomposition\,(\ref{decomposition})\\
\hline
0&$\mathcal{I}(\Delta_{0,1})\cong\mathbb{A}^0$\\
\hline
1&$\mathcal{I}(\Delta_{1,1})\cong\mathbb{A}^0$\\
\hline
2&$\mathcal{I}(\Delta_{2,1})\cong\mathbb{A}^0$, $\mathcal{I}(\Delta_{2,2})\cong\mathbb{A}^1$\\
\hline
3&$\mathcal{I}(\Delta_{3,1})\cong\mathbb{A}^0$, $\mathcal{I}(\Delta_{3,2})\cong\mathbb{A}^1$, $\mathcal{I}(\Delta_{3,3})\cong\mathbb{A}^2$\\
\hline
4&$\mathcal{I}(\Delta_{4,1})\cong\mathbb{A}^0$, $\mathcal{I}(\Delta_{4,2})\cong\mathbb{A}^1$, $\mathcal{I}(\Delta_{4,3})\cong\mathbb{A}^2$, $\mathcal{I}(\Delta_{4,4})\cong\mathbb{A}^2$\\
\hline
5&$\mathcal{I}(\Delta_{5,1})\cong\mathbb{A}^0$, $\mathcal{I}(\Delta_{5,2})\cong\mathbb{A}^1$, $\mathcal{I}(\Delta_{5,3})\cong\mathbb{A}^2$, $\mathcal{I}(\Delta_{5,4})\cong\mathbb{A}^2$, $\mathcal{I}(\Delta_{5,5})\cong\mathbb{A}^3$\\
\hline
6&$\mathcal{I}(\Delta_{6,1})\cong\mathbb{A}^0$, $\mathcal{I}(\Delta_{6,2})\cong\mathbb{A}^1$, $\mathcal{I}(\Delta_{6,3})\cong\mathbb{A}^2$, $\mathcal{I}(\Delta_{6,4})\cong\mathbb{A}^2$, $\mathcal{I}(\Delta_{6,5})\cong\mathbb{A}^3$\\
&$\mathcal{I}(\Delta_{6,6})\cong\mathbb{A}^3$\\
\hline
7&$\mathcal{I}(\Delta_{7,1})\cong\mathbb{A}^0$, $\mathcal{I}(\Delta_{7,2})\cong\mathbb{A}^1$, $\mathcal{I}(\Delta_{7,3})\cong\mathbb{A}^2$, $\mathcal{I}(\Delta_{7,4})\cong\mathbb{A}^2$, $\mathcal{I}(\Delta_{7,5})\cong\mathbb{A}^3$\\
&$\mathcal{I}(\Delta_{7,6})\cong\mathbb{A}^3$\\
\hline
8&$\mathcal{I}(\Delta_{8,1})\cong\mathbb{A}^0$, $\mathcal{I}(\Delta_{8,2})\cong\mathbb{A}^1$, $\mathcal{I}(\Delta_{8,3})\cong\mathbb{A}^2$, $\mathcal{I}(\Delta_{8,4})\cong\mathbb{A}^2$, $\mathcal{I}(\Delta_{8,5})\cong\mathbb{A}^3$\\
&$\mathcal{I}(\Delta_{8,6})\cong\mathbb{A}^3$, $\mathcal{I}(\Delta_{8,7})\cong\mathbb{A}^4$\\
\multicolumn{2}{c}{\textbf {Table\,}{\mathversion{bold}$4$}}
\end{tabular} 
\end{center}

The motivic Hilbert zeta function follows from Table\,4. 
\begin{align}\notag
Z_{C,p}^{\mathrm{Hilb}}(t,\mathbb{L})=&1+t+(1+\mathbb{L})t^2+(1+\mathbb{L}+\mathbb{L}^2)t^3+(1+\mathbb{L}+2\mathbb{L}^2)t^4\\\notag
&+(1+\mathbb{L}+2\mathbb{L}^2+\mathbb{L}^3)t^5
+(1+\mathbb{L}+2\mathbb{L}^2+2\mathbb{L}^3)t^6\\\notag
&+(1+\mathbb{L}+2\mathbb{L}^2+2\mathbb{L}^3)t^7+(1+\mathbb{L}+2\mathbb{L}^2+2\mathbb{L}^3+\mathbb{L}^4)t^8\sum_{i=0}^{\infty}t^i\\\label{MHZ for E_8}
=&\frac{1}{1-t}(1+\mathbb{L}t^2+\mathbb{L}^2t^3+\mathbb{L}^2t^4+\mathbb{L}^3t^5+\mathbb{L}^3t^6+\mathbb{L}^4t^8)
\end{align}

\section{An application}\label{application}
As an application of the results obtained in the previous sections,  
we consider another topic that is related to certain curve counting theories on Calabi-Yau 3-folds in this section.
For the background on these counting theories, refer to \cite{MS} and \cite{PT1}. 
\subsection{Generalization of Kawai's result}
Let $C$ be an integral, locally planar, complex algebraic curve. 
Pandharipande and Thomas \cite{PT2} showed the existences of integers $n_h$ such that 
\begin{equation}\label{BPS}
\sum_{l=0}^{\infty}\chi(C^{[l]})q^{l+1-g}= \sum_{h=\tilde{g}}^gn_h\left\{\frac{q}{(1-q)^2}\right\}^{1-h}
\end{equation}
where $g$ and $\tilde{g}$ are the arithmetic and the geometric genera of $C$ respectively. 
Since $q/(1-q)^2=(q^{\frac{1}{2}}-q^{-\frac{1}{2}})^{-2}$, the left hand side of (\ref{BPS}) is expressed in terms of $(q^{\frac{1}{2}}-q^{-\frac{1}{2}})^{-2}$. 
Actually, for a nonsingular complex projective curve $C$ of arithmetic genus $g$, Macdonald \cite{Macdonald} showed that the Euler numbers $\chi(C^{(l)})$ of the $l$-fold symmetric product $C^{(l)}$ satisfy the following relation:
\begin{equation}\label{Macdonald formula}
\sum_{l=0}^{\infty}\chi(C^{(l)})q^{l+1-g}=(q^{\frac{1}{2}}-q^{-\frac{1}{2}})^{2g-2}
\end{equation}
Here note that $C^{(l)}$ coincides with $C^{[l]}$ for a nonsingular curve $C$.  
Kawai \cite{Kawai} showed that an analogue of (\ref{Macdonald formula}) holds for an integral complex projective curve $C$ of arithmetic genus $g$  
with $m$ $\mathbb{A}_1$-sigularities and $n$ $\mathbb{A}_2$-singularities. 
\begin{equation}\label{Kawai formula}
\sum_{l=0}^{\infty}\chi(C^{[l]})q^{l+1-g}=(q^{\frac{1}{2}}-q^{-\frac{1}{2}})^{2g-2}\left\{1+\frac{1}{(q^{\frac{1}{2}}-q^{-\frac{1}{2}})^2}\right\}^m\left\{1+\frac{2}{(q^{\frac{1}{2}}-q^{-\frac{1}{2}})^2}\right\}^n
\end{equation}
He also discussed the relation between (\ref{Kawai formula}) and the results of \cite{GV}. 

Shende \cite{S} showed that the numbers $n_h$ in (\ref{BPS}) are determined by the local data of the singular points on $C$. 
Actually, under the condition $(*)$, the numbers  $n_h$ depend on motivic Hilbert zeta functions for the singular points on $C$. 
Indeed, it follows from Theorem\,\ref{main 2} that 
\begin{align}\notag
\sum_{l=0}^{\infty}\chi(C^{[l]})q^l&=\left( \sum_{l=0}^{\infty}\chi((C\setminus \coprod_{i}^{s}p_i)^{[l]})q^l\right)\prod_{i=0}^s\left( \sum_{l=0}^{\infty}\chi(C_{p_i}^{[l]})q^l \right)\\ \label{Shende}
&=(1-q)^{2\tilde{g}-2+\sum bi}\prod_{i=0}^{s}Z^{\mathrm{Hilb}}_{C,p_i}(q, 1)
\end{align}
where Sing\,$(C)=\{p_1,\cdots,p_s\}$ and $b_i$ is the number of branches at $p_i$. 
Using (\ref{Shende}), 
we generalize Kawai's result (\ref{Kawai formula})  as follows:
\begin{thm}\label{Generalization of Kawai's Thm}
Let $C$  be an integral complex projective curve of arithmetic genus $g$  
whose singularities consist of the $A_1$, $A_{2i}\ (i=1,\ldots, d)$, $E_6$ and $E_8$-singularities. 
The numbers of the singularities are 
$\#\{A_1\text{-sing.}\}=m_0$, 
$\#\{A_{2i}\text{-sing.}\}=m_i$, 
$\#\{E_{6}\text{-sing.}\}=m_{d+1}$ and
$\#\{E_{8}\text{-sing.}\}=m_{d+2}$. 
There exist polynomials $G_i(T)$ $(i=1,\ldots,d+2)$ in $\mathbb{Z}[T]$ such that
\begin{align*}
\sum_{l=0}^{\infty}\chi(C^{[l]})q^{l+1-g}
&=(q^{\frac{1}{2}}-q^{-\frac{1}{2}})^{2g-2}\left\{1+\frac{1}{(q^{\frac{1}{2}}-q^{-\frac{1}{2}})^2}\right\}^{m_0}
\prod_{i=1}^d\left\{1+\frac{G_i\big((q^{\frac{1}{2}}-q^{-\frac{1}{2}})^{2}+2\big)}{(q^{\frac{1}{2}}-q^{-\frac{1}{2}}\big)^{2i}}\right\}^{m_i}\\
&\times \left\{1+\frac{G_{d+1}\big((q^{\frac{1}{2}}-q^{-\frac{1}{2}})^{2}+2\big)}{(q^{\frac{1}{2}}-q^{-\frac{1}{2}}\big)^{6}}\right\}^{m_{d+1}}
\times \left\{1+\frac{G_{d+2}\big((q^{\frac{1}{2}}-q^{-\frac{1}{2}})^{2}+2\big)}{(q^{\frac{1}{2}}-q^{-\frac{1}{2}}\big)^{8}}\right\}^{m_{d+2}}.
\end{align*}
Especially, we have $G_{d+1}(T)=6T^2-14T+9$ and $G_{d+2}(T)=8T^3-27T^2+33T-15$. 
\end{thm}
\begin{rem}
It follows from  {\rm (\ref{Kawai formula})} that $G_1(x)=2$ in Theorem\,{\rm \ref{Generalization of Kawai's Thm}}. 
In Subsection\,{\rm \ref{Proof of Kawai's Thm}} below, we give an explicit construction of $G_i(x)$ $(i=1,\ldots,d)$. 
\end{rem}

\subsection{Proof of Theorem\,\ref{Generalization of Kawai's Thm}}\label{Proof of Kawai's Thm}
In this subsection, we prove Theorem\,\ref{Generalization of Kawai's Thm}. 
We first define polynomials $F_i(T)\in\mathbb{Z}[T]$ $(i\ge 0)$ as follows:
\begin{equation}
F_0(T):=2,\ F_1(T):=T \ \text{and}\  F_i(T):=TF_{i-1}(T)-F_{i-2}(T) \text{ for $i\ge 2$}
\end{equation}

These polynomials have the following properties:
\begin{lem}\label{w4}
We have $F_i(q+1/q)=q^i+1/q^i$ for $i\ge 1$.
\end{lem}
\begin{proof}
We can prove this assertion by the mathematical induction on $i$. 
\end{proof}

Next, we define polynomials $G_i(x)$ for positive integers $i$. 
Set 
\begin{equation*}
G_1(T):=F_0(T)\ \text{and}\ G_2(T):=4F_1(T)-5.
\end{equation*}
For an  odd number $i=2i'+1\,(\ge 3)$, set
\begin{equation*}
G_i(x):=\sum_{l=0}^{i'-1}\binom{2i}{2l+1}F_{i-(2l+1)}(T)+\sum_{l=1}^{i'}\left\{1-\binom{2i}{2l}\right\}F_{i-2l}(T)+\binom{2i}{i}. 
\end{equation*}
On the other hand, for an even number $i=2i'\, (\ge 4)$, set
\begin{equation*}
G_i(x):=\sum_{l=0}^{i'-1}\binom{2i}{2l+1}F_{i-(2l+1)}(T)+\sum_{l=1}^{i'-1}\left\{1-\binom{2i}{2l}\right\}F_{i-2l}(T)+1-\binom{2i}{i}. 
\end{equation*}
\begin{prp}\label{w2}
For a positive integer $i$, we have
\begin{align*}
\frac{\sum_{l=0}^iq^{2l}}{(1-q)^{2i}}=1+\left\{\frac{q}{(1-q)^{2}}\right\}^iG_i\Big(q+\frac{1}{q}\Big).
\end{align*}
\end{prp}
\begin{proof}
For any positive integer $i$,  the relation
$$
\sum_{l=0}^iq^{2l}-(1-q)^{2i}=q^iG_i\Big(q+\frac{1}{q}\Big)
$$
follows from the binomial expansion of $(1-q)^{2i}$ and Lemma\,\ref{w4}. 
It implies our desired result.
\end{proof}
\noindent
\textit{Proof of Theorem\,{\rm \ref{Generalization of Kawai's Thm}}.} 
Let $p_0$, $p_i$ $(i=1,\dots,d)$, $p_{d+1}$ and $p_{d+2}$ be $A_1$, $A_{2i}$, $E_6$ and $E_8$-singularities on $C$ respectively. 
In our situation, (\ref{Shende}) becomes
\begin{equation}\label{w1}
\sum_{l=0}^{\infty}\chi(C^{[l]})q^l=(1-q)^{2\tilde{g}-2+2m_0+\sum m_i}\prod_{i=0}^{d+2}\left (\sum_{l=0}^{\infty}Z_{C,p_i}^{\mathrm{Hilb}}(q,1)\right )^{m_i}. 
\end{equation}
We have $\tilde{g}=g-m_0-\sum m_i\delta_i$ where $\delta_i$ is the $\delta$-invariant of $(C,p_i)$. 
Here recall that $\delta_i=i$ for $i=1,\ldots,d$, $\delta_{d+1}=3$ and $\delta_{d+2}=4$. 
It follows from (\ref{MHZ for A_1}), (\ref{MHZ for A_{2d}}) , (\ref{MHZ for E_6}), (\ref{MHZ for E_8}) and (\ref{w1}) that
\begin{align}\notag
&(1-q)^{2\tilde{g}-2+2m_0+\sum m_i}\left\{\frac{1-q+q^2}{(1-q)^2}\right\}^{m_0}\prod_{i=1}^{d}\Bigg (\frac{\sum_{l=0}^{i}q^{2l}}{1-q}\Bigg )^{m_i}\\\notag
&\qquad\qquad\qquad\qquad\qquad\qquad\times \Bigg (\frac{\sum_{l=0,l\neq 1,5}^{6}q^{l}}{1-q}\Bigg )^{m_{d+1}}
\times \Bigg (\frac{\sum_{l=0,l\neq 1,7}^{8}q^{l}}{1-q}\Bigg )^{m_{d+2}}\\ \notag
&=(1-q)^{2(g-m_0-\sum m_i\delta _i)-2}(1-q+q^2)^{m_0}\prod_{i=1}^{d}\Big (\sum_{l=0}^{i}q^{2l} \Big)^{m_i}\\\notag
&\qquad\qquad\qquad\qquad\qquad\qquad\times \Big (\sum_{l=0, l\neq 1,5}^{6}q^{l} \Big)^{m_{d+1}}
\times \Big (\sum_{l=0, l\neq 1,7}^{8}q^{l} \Big)^{m_{d+2}}\\ \notag
&=(1-q)^{2g-2}\left\{\frac{1-q+q^2}{(1-q)^2}\right\}^{m_0}\prod_{i=1}^{d}\Bigg \{\frac{\sum_{l=0}^{i}q^{2l}}{(1-q)^{2i}}\Bigg \}^{m_i}\\\label{w3}
&\qquad\qquad\qquad\qquad\qquad\qquad\times \Bigg \{\frac{\sum_{l=0,l\neq 1,5}^{6}q^{l}}{(1-q)^6}\Bigg \}^{m_{d+1}}
\times \Bigg \{\frac{\sum_{l=0,l\neq 1,7}^{8}q^{l}}{(1-q)^8}\Bigg \}^{m_{d+2}}. 
\end{align}

By direct calculations, we can check that
\begin{align}\label{w4}
\frac{1-q+q^2}{(1-q)^2}&=1+\frac{q}{(1+q)^2},\\\label{w5}
\frac{\sum_{l=0,l\neq 1,5}^{6}q^{l}}{(1-q)^6}&=1+\frac{q^3}{(1-q)^6}\left\{6\left(q+\frac{1}{q}\right)^2-14\left(q+\frac{1}{q}\right)+9\right\},\\\label{w6}
\frac{\sum_{l=0,l\neq 1,7}^{8}q^{l}}{(1-q)^8}&=1+\frac{q^4}{(1-q)^8}\left\{8\left(q+\frac{1}{q}\right)^3-27\left(q+\frac{1}{q}\right)^2+33\left(q+\frac{1}{q}\right)-15\right\}. 
\end{align}
Set $G_{d+1}(T):=6T^2-14T+9$ and $G_{d+2}(T):=8T^3-27T^2+33T-15$. 
It follows from (\ref{w3}), (\ref{w4}), (\ref{w5}), (\ref{w6}) and  Proposition\,\ref{w2} that
\begin{align*}
\sum_{l=0}^{\infty}\chi([C^{[l]}])q^{l+1-g}&=\left\{\frac{(1-q)^2}{q}\right\}^{g-1}\left\{1+\frac{q}{(1-q)^2}\right\}^{m_0}
\prod_{i=1}^{d}\left\{1+\frac{q^i}{(1-q)^{2i}}G_i\Big(q+\frac{1}{q}\Big)\right\}^{m_i}\\
&\times \left\{1+\frac{q^3}{(1-q)^{6}}G_{d+1}\Big(q+\frac{1}{q}\Big)\right\}^{m_{d+1}}
\times \left\{1+\frac{q^4}{(1-q)^{8}}G_{d+2}\Big(q+\frac{1}{q}\Big)\right\}^{m_{d+2}}. 
\end{align*}

Since the relations $(q^{\frac{1}{2}}-q^{-\frac{1}{2}})^2=q-2+1/q=(1-q)^2/q$ hold, the theorem is proved.  $\square$

\noindent
\small 
Masahiro Watari\\
 University of Kuala Lumpur, Malaysia France Institute\\
Japanese Collaboration Program\\
Section 14, Jalan Damai, Seksyen 14, 43650\\
Bandar Baru Bangi, Selengor, Malaysia.\\
E-mail:masahiro@unikl.edu.my
\end{document}